\documentclass[11pt]{amsart}

\usepackage[demo]{graphicx}
\usepackage{subcaption} 

\usepackage{tikz} 
\usepackage{wrapfig}
\usetikzlibrary{arrows,shapes,positioning}
\usetikzlibrary{decorations.markings}
\tikzstyle arrowstyle=[scale=1]  
\tikzstyle directed=[postaction={decorate,decoration={markings,
    mark=at position .65 with {\arrow[arrowstyle]{stealth}}}}]
\tikzstyle reverse directed=[postaction={decorate,decoration={markings,
    mark=at position .65 with {\arrowreversed[arrowstyle]{stealth};}}}]
\usepackage{float} 
\usepackage[margin=1.0in]{geometry} 
\usepackage[skip=2pt,font=scriptsize]{caption}

\usepackage{amsmath,amsfonts,amsthm,amssymb,graphics,color} 

\usepackage{hyperref}   
\usepackage[utf8]{inputenc}
\usepackage[T1]{fontenc}
\usepackage{enumerate}

\newcommand{\black}[1]{\color{black}}
        \definecolor{pink}{rgb}{1,0,1}

\usepackage[abbrev]{amsrefs}

\newtheorem{theorem}{Theorem}
\newtheorem{prop}[theorem]{Proposition}
\newtheorem{lemma}[theorem]{Lemma}
\newtheorem{cor}[theorem]{Corollary}

\newtheorem{remark}[theorem]{Remark}
\theoremstyle{definition}
\newtheorem{defn}[theorem]{Definition}

\numberwithin{equation}{section}

\newcommand{\supp}{\operatorname{Supp}}

\newcommand{\pa}{\partial}

\newcommand{\Sings}{\operatorname{Sing Supp}}

\newcommand{\cL}{{\mathcal{L}}}

\newcommand{\cB}{{\mathcal{B}}}

\newcommand{\R}{\mathbb{R}}

\DeclareMathOperator{\Tr}{Tr}

\title{The Dirichlet isospectral problem for trapezoids} 
\author[Hamid Hezari]{Hamid Hezari}\address{Department of Mathematics, 510J Rowland Hall - University of California, Irvine, CA 92697-3875.} \email{hezari@uci.edu}

\author[Zhiqin Lu]{Zhiqin Lu}\address{Department of Mathematics, 410D Rowland Hall - University of California, Irvine, CA 92697-3875.} \email{zlu@uci.edu}

\author[Julie Rowlett]{Julie Rowlett} \address{Mathematics Department, Chalmers University and the University of Gothenburg, 41296, Gothenburg Sweden} \email{julie.rowlett@chalmers.se} 


\begin{document}

\begin{abstract}
We show that non-obtuse trapezoids are uniquely determined by their Dirichlet Laplace spectrum. This extends our previous result \cite{hlrn}, which was only concerned with the Neumann Laplace spectrum.
\end{abstract}

\maketitle
%\tableofcontents

%%%%%%%%%%%%%%%%%%%%%%%%%%%%%%%%%%%%%%%%%%%%%%%%%%%%%%%%%%%%%%%%%%%%%%%%%%%%%%%%%%%%

\section{Introduction}
 M. Kac  popularized the isospectral problem for planar domains with a paper \cite{kac} titled ``Can one hear the shape of a drum?''   For a bounded, connected domain $\Omega$ in $\R^2$, we define $\Delta_\Omega^\mathcal{B}$ to be the Laplace operator on $\Omega$ with boundary condition $\mathcal{B}$, where $\mathcal{B}$ is either Dirichlet or Neumann.  We consider the Laplace eigenvalue equation, 
\begin{equation*} \label{lap} \Delta_\Omega ^\mathcal B u= - \frac{\pa^2 u}{\pa x^2} - \frac{\pa^2 u }{\pa y^2} = \lambda u, \quad \mathcal B(u) = 0 \; \textrm{ on the boundary of } \Omega. \end{equation*} 
The eigenvalues form a discrete subset of $[0, \infty)$,  $0\leq \lambda_1 < \lambda_2 \leq \lambda_3 \leq \cdots$.  If one takes the Dirichlet boundary condition, requiring the function $u$ to vanish at the boundary, then the set of eigenvalues, known as the spectrum of $\Omega$, is in bijection with the resonant frequencies a drum would produce if $\Omega$ were its drumhead.  With a perfect ear  one could hear all these frequencies and therefore know the spectrum.  Kac's question mathematically means:  \em if two such domains are isospectral, then are they isometric? \em  Gordon, Webb, and Wolpert answered Kac's question in the negative \cites{gww,gww1} (see also \cite{chapman} for an accessible presentation).  All the known counterexamples to date consist of non-covex polygons. 
% JR:  if we use double quotes one place and single quotes another place it seems weird.  So changed to italics for the paraphrasing.  

% JR:    Two sentences is a bit short to stand by itself as a paragraph.  

% JR:  very large class of domains:  removed smoothly bounded because Kac's result does not require this assumption.  

On the other hand, in certain settings this isospectral question can have a positive answer. There are many types of positive results. One question is whether any domain is spectrally unique (up to rigid motions) among a very large class of domains. In this direction, Kac proved that disks can be heard among all domains. He used the heat trace invariants to prove that the area and perimeter of a domain are determined by its spectrum, so by the isoperimetric inequality, disks are spectrally determined. 
Watanabe \cite{wat2} proved that there are certain nearly circular oval domains that are spectrally unique. Recently, the first author and Zelditch \cite{hz2} showed that one can hear the shape of nearly circular ellipses among all smooth domains. A  weaker inverse spectral problem is to find domains that are locally spectrally unique, meaning that that they can be heard among nearby domains in a certain topology.  Marvizi and Melrose \cite{mm} constructed a two-parameter family of planar domains that are locally spectrally unique in the $C^\infty$ topology. The two parameter family consists of domains that are defined by elliptic integrals, and that resemble ellipses but are not ellipses.  For more on positive inverse spectral problems we refer the readers to the surveys \cites{dh, zelsurvey}. 

The notion of spectral rigidity of a domain $\Omega$ is even weaker than local spectral uniqueness. It means that any $1$-parameter family of isospectral domains containing $\Omega$ and staying within a limited class, must be trivial, i.e. made out of rigid motions. In this setting, Popov and Topalov \cite{pt} have recently shown that ellipses are spectrally rigid within the class of analytic domains with the two axial symmetries of an ellipse. In a recent article \cite{hz2}, using a length spectral rigidity theorem of de Simoi, Kaloshin, and Wei \cite{dkw}, the first author and Zeldtich have proven that nearly circular domains with one axis of symmetry are spectrally rigid among such domains.  The final setting is {infinitesimal spectral rigidity} of a given domain $\Omega$, which requires that the first variation of any non-trivial $1$-parameter family of isospectral domains containing $\Omega$ and staying within a class vanishes. It was proved in \cite{hz1} that ellipses are infinitesimally spectral rigid among smooth domains with the axial symmetries of an ellipses. 

%\red{the above paragraph maybe too long}  JR:  broke it up into two paragraphs that split naturally.  

 %\red{I don't understand the following sentence} 
Another interesting setting, into which our result fits, is when one tries to show that the Laplace spectrum map is one-to-one in a relatively small class of domains. The class of domains is either infinite dimensional, in which case normally a generic property is added to simplify an otherwise difficult problem, or it is finite dimensional where no genericity assumption is imposed.  In the former setting, Zelditch \cite{zel} proved that generic analytic domains with an axial symmetry are spectrally distinguishable from each other.  The few inverse problems to date that consider a finite dimensional class of planar domains are about polygonal domains.  The first result of this type is due to Durso \cite{durso}, who proved that the shape of a triangle can be heard among other triangles.   Using the spectral invariants obtained from the short time asymptotic expansion of the heat trace, any two triangles that  are isospectral must have the same area and perimeter.  Since triangles depend on 3 independent parameters, to obtain her result, Durso  used another spectral invariant, namely, the wave trace.  She demonstrated that the length of the shortest closed geodesic in a triangular domain is also a spectral invariant.  More recently Grieser and Maronna \cite{grma} realized that if one used an additional spectral invariant from the heat trace, then this together with the area and perimeter uniquely determine the triangle; that is a much simpler proof.

%\red{In the above paragraph, I separate finite dimensional problems with infinity ones. }

%Before we state more results and problems on polygonal domains, which is the main interest of this article, for the sake of completeness of the presentation, let us also recall two other types of inverse spectral problems.

%We now focus on the class of polygonal domains other than triangles:  JR:  revised.  

After triangles, it is natural to consider quadrilaterals.  For rectangles, it is straightforward to prove that if two rectangles are isospectral, then they are congruent.  In fact, one only requires the first two Dirichlet eigenvalues to prove this.  For parallelograms, it is also a straightforward argument using the first three heat trace invariants as in \cite{sos} to prove that isospectral parallelograms are congruent.  The next natural generalization is to trapezoids.  In this case, one can prove that the geometric information that can be extracted from the heat trace is insufficient to prove that isospectral trapezoids are congruent.  It is therefore necessary to use the wave trace in the spirit of \cite{durso}, which is a much more delicate matter.  The wave trace is a tempered distribution that is a spectral invariant.  To use the wave trace in an isospectral problem, one studies the times at which the wave trace is singular.  For smoothly bounded domains, \cite{gm79} showed that the set of positive times at which the wave trace is singular is contained in the set of lengths of closed geodesics; this is known as \em the Poisson relation.  \em  Once the boundary is no longer smooth, this relation is only known to hold in certain geometric settings.   Here we rely on the work of \cites{Wu, Hil1} to obtain the Poisson relation.  However, the Poisson relation is only a containment.  To be able to state that the length of a certain closed geodesic is a spectral invariant, one must study the singularity in the wave trace at time equal to that length.  Hence, exploiting this technique requires not only careful study of the wave trace but also detailed information on the closed geodesics in the domain.  

The study of closed geodesics in polygonal domains has quite a long history, that to the best of our knowledge was initiated by Fagnano in 1775. Fagnano proved that the orthic triangle (also called Fagnano triangle) is the shortest closed geodesic inside an acute triangle.  Two centuries passed before Schwartz demonstrated the existence of closed geodesics in certain obtuse triangular domains \cites{schwartz, schwartz2}.  We refer to the survey article of Gutkin \cite{g2} for what is known about the existence and distribution of closed geodesics in polygonal domains. 

Our main result is the following. 
\begin{theorem} \label{th:main}  Let $T_1, T_2 \subset \R^2$ be two non-obtuse  trapezoidal domains.  Then if the spectra of the Euclidean Laplacian with Dirichlet boundary condition on $T_1$ and $T_2$ coincide, the trapezoids are congruent, that is equivalent up to a rigid motion of the plane. 
\end{theorem}

In \cite{hlrn}, we proved this result when Neumann boundary condition is considered. One of the key elements of our proof was the use of the singularity in the wave trace that is produced by the orbit $2b$, the orbit that bounces between the the top two vertices of the trapezoid; see Figure \ref{b}. In the Neumann case, we were able to compute and use the leading term of the singularity expansion of the wave trace at $t=2b$.  In the Dirichlet case this singularity has a lower order, and the computation of its leading term is much more complicated. In this paper we avoid doing such computations and accomplish the Dirichlet case by carefully studying other periodic orbits. In fact our proof includes the Neumann case as well and is not special to the Dirichlet case. Our proof still uses several key results from \cite{hlrn} on the wave trace singularity expansions associated to certain diffractive orbits.

\subsection{Organization of the paper} 
In Section \ref{s:heattrace} we present the heat trace invariants and their use and limitations in the determination of a trapezoid. Section \ref{s:wavetrace} introduces the wave trace and the Poisson relation for polygons. We also define the order of a wave trace singularity. In Section \ref{s: closed geodesics}, we specialize to the case of a trapezoid and study the important periodic orbits that we need for our argument, together with a thorough analysis of their singularity contribution. The proof of our main theorem is given in Section \ref{s:proof}.

%%%%%%%%%%%%%%%%%%%%%%%%%%%%%%%%%
\section{Heat trace invariants of trapezoids} \label{s:heattrace} 
In this section, we present a small collection of geometric spectral invariants that can be obtained through the asymptotic behavior of the heat trace as $t \downarrow 0$.  For the sake of completeness and to set the notation, we define the parameters of a trapezoid.   

\begin{defn}
A \em trapezoid \em is a convex quadrilateral that has two parallel sides of lengths $b$ and $B$ with $B\geq b$. The side of length $B$ is called the \em base. \em The two angles $\alpha$ and $\beta$ adjacent to the base are called base angles.  A trapezoid is called non-obtuse if the base angles satisfy
\[
0<\beta\leq\alpha\leq \frac\pi 2.  
\]
We make this assumption throughout the paper.  If $\alpha=\beta$, then we say the trapezoid is \em isosceles \em. The other two sides of the trapezoid are  known as \em legs \em of lengths $\ell$ and $\ell'$, respectively.  The distance between two parallel sides is called the \em height. \em See Figure~\ref{trap} for a picture of a trapezoid.  
\end{defn}

\begin{figure}[H]
	\begin{tikzpicture}
	\draw[thick] (0,0) -- (4,0)  -- (2.5,1.5)  -- (.5,1.5) -- (0,0) ; 
	\node at (2,0) [align=center, below]{$B$};
	\node at (1.5,1.5) [align=center, above]{$b$};
	\node at (.05, .8) [align=center]{$\ell$};
	\node at (3.55, .9) [align=center]{$\ell'$};
	\node at (.25, .15) [align=center]{$\alpha$};
	\node at (3.5, .15) [align=center]{$\beta$};
	\draw[red] (1.5,0)--(1.5,1.5);
	\node at (1.7, .8) [align=center]{$h$};

	\end{tikzpicture} 
	\centering
	\caption{A non-obtuse trapezoid and its parameters.} 
	\label{trap} 
\end{figure}
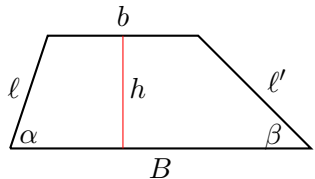%

We consider a trapezoid $T$ as in Figure~\ref{trap}. We use $b$, $B$, $\ell$, and $\ell'$ to denote the lengths of the shorter and longer parallel sides, respectively, and the lengths of the two legs of the trapezoid. Abusing notation, we may also use these to denote the corresponding edges. 

\begin{defn} Any quantity that is uniquely determined by the spectrum is known as a \em spectral invariant.  \em  Colloquially and in the spirit of \cite{kac}, we say that ``$X$ can be heard,'' if the quantity $X$ is a spectral invariant.  
\end{defn} 

 \subsection{The heat trace} 
Let $\{ \lambda_k \}_{k \geq 1}$ denote the eigenvalues. We define 
\[ \Tr e^{-t\Delta^{\mathcal B}_\Omega } = \sum_{k \geq 1} e^{-t \lambda_k } \] 
to be  the trace of the heat kernel, which is the Schwartz kernel of the fundamental solution to the heat equation.  The heat trace, which is a spectral invariant, is an analytic function for $\Re(t) > 0$.  It is well known in this setting (see \cites{kac,ms,vdb}) that the heat trace on a polygonal domain $\Omega$ admits an asymptotic expansion as $t \downarrow 0$, 
\begin{equation}\label{1}
\Tr e^{-t\Delta^{\mathcal B}_\Omega } = \frac{|\Omega|}{4 \pi t} \pm \frac{| \pa \Omega|}{8 \sqrt{\pi t}} + \sum_{k=1} ^n \frac{\pi^2 - \theta_k ^2}{24 \pi \theta_k} + O(e^{-c/t}), \quad t \downarrow 0.
\end{equation}
Above, $|\Omega|$ and $|\pa \Omega|$ denote respectively the area and perimeter of the domain $\Omega$, and $\theta_k$ are the interior angles. In the second term we choose the minus sign when $\mathcal B = \text{Diriclet}$ and the plus sign when  $\mathcal B = \text{Neumann}$.   The constant $c>0$ has been estimated in \cite{vdb}.  Since the angles of a trapezoid are $\alpha$, $\pi - \alpha$, $\beta$, and $\pi - \beta$, we therefore have the following.  

\begin{prop}  \label{qprop} For a trapezoidal domain, the area $A$, perimeter $L$, and the angle invariant $q$ defined by
\begin{equation} \label{dqfa} 
q:=F(\alpha) + F(\beta), \quad F(x) := \frac{1}{x (\pi - x)}.
\end{equation}
are spectral invariants.  
\end{prop}

\begin{prop} \label{p:rect} The spectrum determines whether or not a trapezoid is a rectangle.  If one trapezoid is a rectangle and is isospectral to another trapezoid, then that trapezoid is also a rectangle, and the two rectangles are congruent.  %the same. 
\end{prop} 
\begin{proof} 
	Note that if we rewrite the angle invariant as
	\[ q = \frac{1}{\pi} \left ( \frac{1}{\alpha} +  \frac{1}{\pi - \alpha}  + \frac{1}{\beta}   + \frac{1}{\pi -\beta} \right ),\]  then by the `arithmetic mean-harmonic mean' inequality, we have
	\begin{equation*}\label{iso-17}
	q \geq \frac{8}{\pi^2},
	\end{equation*}
	and equality holds if and only if the trapezoid is actually a rectangle. 
If two trapezoids are isospectral, then they have the same value of $q$.  Hence they are either both rectangles or neither is a rectangle.  If they are both rectangles, and they are isospectral, then the dimensions of the rectangles can be obtained by the first two eigenvalues.  These uniquely determine the rectangle up to congruency, that is up to rigid motions of the plane. 
\end{proof}

Since the moduli space of trapezoids is four dimensional, the three heat trace invariants introduced above cannot determine the shape of the trapezoids.  To extract more information from the spectrum we turn to the wave trace.

%%%%%%%%%%%%%%%%%%%%%%%%%%%%%%%%%%%%%%%%%%

\section{Poisson relation and singularities of the wave trace for polygons} \label{s:wavetrace}
In this section, we study the singularities of the wave trace of the Laplacian on a polygon.  The wave trace is the trace of the wave propagator, also known as the trace of 
the wave group, and can be written as  
\[ w_\Omega ^\cB (t):=\Tr e^{i t\sqrt{\Delta_\Omega ^{\mathcal B}} } = \sum_{k \geq 1} e^{it
\sqrt{\lambda_k} }. \] 
The wave trace is only well-defined when paired with a Schwartz class test function; it is a tempered distribution by Weyl's law. The connection between the wave trace and geodesic trajectories comes from the fact that the singularities of the wave operator propagate along geodesic trajectories.  For smoothly bounded domains, the times at which the wave trace is singular is contained in the set of lengths of generalized broken periodic geodesics \cite{PeSt}*{Theorem 5.4.6}; see also ~\cites{am,gm79}. 
Propagation of singularities of the wave operator in a polygonal domain is more difficult to study because of diffraction phenomena that may occur at the vertices. 

One way to study the wave trace on a polygonal domain is to double it and create an associated \em euclidean surface with conical singularities, \em or ESCS as in \cite{hil}. An ESCS is a compact manifold with finitely many conical singularities that is locally flat away from the conical points,  and near the conical points, it is isometric to a neighborhood of the vertex of a euclidean cone.  It was shown separately by Hillairet \cite{Hil1} and Wunsch \cite{Wu} that the positive singular support of the wave trace for the Friedrichs extension of the Laplacian on an ESCS is contained in the set of lengths of periodic geodesics on the ESCS.  On an ESCS, the conical points are separated into two groups.  A conical point on an ESCS is \em non-diffractive \em if its angle is equal to $\frac{2 \pi}{n}$ for some positive integer, $n$, otherwise it is called \em diffractive.  \em  

\begin{defn} A closed geodesic in a polygonal domain is a geodesic trajectory that forms a closed, piecewise linear curve that bounces off the edges according to the equal angle law.  We say that a closed geodesic is conical if it meets at least one vertex.  
When a geodesic trajectory meets a vertex,  if the interior angle at a vertex is of the form $\frac \pi N$ for positive integer $N>1$, then noting that the upper half space is an $N$-fold covering, this specifies the angle at which the trajectory leaves that vertex.  However, when the interior angle at a corner is not of the form $\frac \pi N$ for any positive integer $N>1$, Keller's democratic law of diffraction states that a billiard trajectory that hits the corner departs that corner in every direction \cite{keller}.  Such corners are called \em diffractive.  \em  The other corners are \em non-diffractive.  \em  Similarly, geodesics are classified as diffractive if they meet at least one diffractive corner, otherwise they are non-diffractive. 
 \end{defn} 

By our definitions,  all non-conical closed geodesics are non-diffractive geodesics.  However, because non-diffractive geodesics may pass through vertices with angles of the form $\frac \pi N$, not all non-diffractive closed geodesics are non-conical.

For a polygonal domain $\Omega \subset \R^2$, define 
\[\cL{sp} (\Omega) := \left \{ \textrm{lengths of closed diffractive or non-diffractive geodesics in $\Omega$} \right\}.\] 
Let 
\[ \Sings w_\Omega ^\mathcal B(t) := \textrm{ the singular support of } w_\Omega ^\mathcal B(t). \]

The utility of the length spectrum follows from the Poisson relation.
\begin{theorem}[Poisson relation for polygons \cites{Hil1, durso, Wu}] For a polygonal domain $\Omega$, we have 
\[ \Sings w_\Omega ^\cB (t) \subset  \{0\}  \cup { \pm \cL{sp} (\Omega)}. \] 
This holds for both the Dirichlet and Neumann boundary condition.  
\end{theorem} 

By a compactness argument one can prove %show that,  JR:  the sentence reads By a compactness argument one can prove Lemma 8.  Not as good to have `one can show that Lemma 8.'  
\begin{lemma} \label{le:katok} There are no accumulation points in $\cL{sp} (\Omega)$ for any polynomial $\Omega$.
	\end{lemma}
	In other words, the length spectrum of  a  polygon is a discrete set in $[0, \infty)$. %\red{Can we write $[0,\infty)$ instead ?} 
	However,  this lemma also follows immediately from a much stronger result of Katok, namely 
	\begin{theorem}  \cite{katok}
	The counting function of the lengths of geodesics starting and ending at vertices of a polygonal domain is of sub-exponential growth. 
	\end{theorem}

By the Poisson Relation together with Lemma \ref{le:katok}, the singularities of the wave trace are discrete.  Consequently we can enumerate the singularities and for example speak about the shortest, or the second shortest positive singularity.  It is also possible to find test functions that are supported in a neighborhood of one and only one singular time.  This allows us to define the order of a singularity.   

\begin{defn} Suppose $t_0>0$ is in the singular support of  $w_\Omega ^\mathcal B(t)$.  Let $\hat \rho(t)$ be a cutoff function supported in a neighborhood of $t_0$, such that $\hat \rho \equiv 1$ near $t_0$.  Assume that 
\[ \supp \hat \rho \cap \Sings w_\Omega ^\mathcal B= \{t_0\}. \] 
We define \em the frequency domain contribution of the singularity $t_0$ \em by
\[  I_{t_0, \mathcal{B}}(k) := \int_\R \hat \rho(t) e^{-ikt}\, \text{Tr}\, e^{i t\sqrt{\Delta^\mathcal B_{\Omega}}} dt. \]  
We say that a singularity $t_0$ is of order $a \in \R$ if 
\[ I_{t_0, \mathcal{B}}(k) = c k^a + o(k^a), \quad k \to \infty. \] 
Above, $c$ is a constant that depends on the microlocal germ of the domain near the closed geodesics of length $t_0$.
We say $t_0$ is \textit{at most of order $a$} if 
\[ I_{t_0, \mathcal{B}}(k) = O (k^a). \] 
\end{defn}

\begin{remark} \label{order} We note that near a singularity at $t=t_0$ of order $a$ as defined above, the wave trace belongs to $H^{-s}(\R)$ for all $s > a+\frac{1}{2}$, but does not belong to $H^{-s}(\R)$ for $s=a+\frac{1}{2}$.
\end{remark}

To use the Poisson relation, we investigate the shortest closed geodesics in trapezoids.  Parallel families of closed geodesics play a central role.

\section{Closed geodesics inside a trapezoid and their singularity contribution} \label{s: closed geodesics}
We start this section by recalling some standard facts about periodic orbits inside a polygon from Gutkin (see \cites{g1, g2}). 
It is important to note that in the dynamical systems literature periodic orbits inside a polygon refer to our non-conical closed geodesics, i.e. closed geodesics that do not hit any vertices of the polygon. Also, {generalized periodic orbits} refer to what we call {geometric conical geodesics}, which are precisely limits of non-conical closed geodesics. Diffractive periodic orbits that are not geometric (see for example Figures \ref{b}, \ref{h alpha}) are not considered in the purely dynamical systems references, but are of great interest in PDE because of their contribution to the singularities of solutions to the wave equation. 
\\

We start by defining prime periodic  orbits.  

\begin{defn}
A non-conical periodic orbit is called prime if it is not a multiple of another one.
\end{defn}
We then have the following classification. 
\begin{prop} [Gutkin \cite{g2}, Corollary 1]
	 Let $g$ be a prime non-conical closed geodesic of period $n$ in a polygon $P$.  Here, period refers to the number of times the orbit meets the edges.
	 \begin{enumerate} 
	 	\item If $n$ is even, then $g$ is contained in a band of parallel periodic orbits, all of the same length. Let $S$ be the maximal band containing $g$. Then $S$ is a closed flat cylinder. Each of the boundary circles of $S$ is a conical geodesic of $P$. 
	 	\\
	 	
	 	\item If $n$ is odd, then the orbit $g$ is isolated. The maximal strip $S$ of periodic orbits parallel to $g$ is a flat M\"obius band, and $g$ is the middle circle of $S$. Precisely, $S$ is the union of the periodic orbits of length twice the length of $g$ that are parallel to $g$. The boundary circle of $S$ is a conical geodesic. 
	 	 \end{enumerate} 
	\end{prop}

This proposition makes an important distinction between the prime periodic billiard orbits of odd and even periods. The former are isolated; the latter form periodic cylinders, hence are never isolated. Cylinders of periodic orbits cause a larger singularity in the wave trace.  Let us now discuss some examples.

\subsection{Important examples of closed geodesics inside a trapezoid}
Here, we list only examples of closed geodesics inside a trapezoid that are key to our argument of the main theorem. See Figures 
\ref{h}, \ref{orthic}, \ref{b}, \ref{h alpha}. We postpone the study of their wave trace singularity contributions to the next section.
\begin{figure} [H]
	\centering
	\begin{minipage}{.5\textwidth}
		\begin{tikzpicture}[scale=0.55]
		\draw[thick ] (0,0) -- (5.8,0)  -- (4.75, 2.4)  -- (.65,2.4) --(0,0) ; 
		\draw  node at(0.65,2.4) [draw=red,line width=0.1pt, circle,fill=blue!50,inner sep=0pt,minimum size=2pt,line width=1pt] {} ;   
		\draw[very  thick, red!80]  (.65,2.3)--(0.65,0);
		\draw[very  thick, red!80]  (4.75,2.3)--(4.75,0);
		\draw  node at(4.75,2.4) [draw=red,line width=0.1pt, circle,fill=blue!50,inner sep=0pt,minimum size=2pt,line width=1pt] {} ;   
		\draw[very  thick, blue!70]  (1.333,2.4)--(1.333,0);
		\draw[very  thick, blue!70]  (2.01633,2.4)--(2.016333,0);
		\draw[very  thick, blue!70]  (2.7,2.4)--(2.7,0);
		\draw[very  thick, blue!70]  (3.383,2.4)--(3.383,0);
		\draw[very  thick, blue!70]  (4.066,2.4)--(4.066,0);
		\end{tikzpicture} 
		\centering
		\caption{The $2h$ family.}
		\label{h}
	\end{minipage}%
	\begin{minipage}{.5\textwidth}
		\begin{tikzpicture}[scale=0.45]
		\draw[thick ] (0,0) -- (6,0)  -- (2.833,3.6)  -- (1.2,3.6) --(0,0) ; 
		\draw[dash pattern=on 1.5pt off 1pt on 1.5pt off 1pt, red] (1.8,0) -- (1.8, 3.6)   ; 
		\draw[dash pattern=on 1.5pt off 1pt on 1.5pt off 1pt, red] (.7,1.9) -- (6,0)   ; 
		\draw[dash pattern=on 1.5pt off 1pt on 1.5pt off 1pt, red] (3.5,2.9) -- (0,0)   ; 
		\draw[very thick,blue! 70] (3.5,2.85) -- (1.8,0)--(.67,1.93)--(3.5,2.85)   ; 
		\end{tikzpicture} 
		\centering
		\caption{The $l_F$ orbit. } 
		\label{orthic}
	\end{minipage}
	
	\vskip\baselineskip
	\begin{minipage}{.5\textwidth}
		\begin{tikzpicture}[scale=0.70]
		\node at (1.3, 2.3) [align=center]{$b$};
		\draw[thick ] (0,0) -- (3.8,0)  -- (2., 2.)  -- (.65,2.) --(0,0) ; 
		\draw  node at(0.65,2.) [draw=red,line width=0.1pt, circle,fill=blue!50,inner sep=0pt,minimum size=2pt,line width=1pt] {} ;   
		\draw[line width=.5mm, blue](2., 2.)  -- (.745,2.) ;
		
		\draw  node at(2,2.) [draw=red,line width=0.3pt, circle,fill=blue!50,inner sep=0pt,minimum size=2pt,line width=1pt] {} ;   
		\end{tikzpicture} 
		\centering
		\caption{The $2b$ orbit.}
		\label{b}
	\end{minipage}%
	\begin{minipage}{.5\textwidth}
		\begin{tikzpicture}[scale=0.50]
		\draw[thick ] (0,0) -- (6,0)  -- (3,3)  -- (1,3) --(0,0) ; 
		\draw  node at(0, 0) [draw=red,line width=0.1pt, circle,fill=blue!50,inner sep=0pt,minimum size=2pt,line width=1pt] {} ;   
		\draw[very  thick, blue!70]  (3.3,2.7) -- (0,0)  ;
		\end{tikzpicture} 
		\centering
		\caption{The $2h_\alpha$ orbit. } 
		\label{h alpha}
	\end{minipage}
	
\end{figure}

\textbf{The $2h$ family.} It consists of the bouncing ball orbits parallel to the height of the trapezoid. It is a cylinder of periodic orbits of order $n=2$. The area it sweeps  is $2hb$; see Figure \ref{h}
\\

\textbf{The Fagnano orbit (or $l_F$) {and its double}.}  This orbit is also called the \emph{orthic} triangle and  is the triangle that joins the feet of the altitudes of the extended triangle of the trapezoid; see Figures \ref{orthic} and \ref{trapezoid}.  The Fagnano orbit exists only if the extended triangle is acute, and if the height of the trapezoid is not too short. In fact one can easily see that this condition on $h$ is
\begin{equation} \label{lF exists} h \geq \max \{B \sin \alpha \cos \alpha, B \sin \beta \cos \beta \} = B \sin \beta \cos \beta. \end{equation}
The last equality happens because the extended triangle of the trapezoid being acute requires that $\alpha + \beta \geq \frac \pi 2$, which implies that $\sin 2\alpha \leq \sin 2 \beta$. The length of the Fagnano orbit is given by
\[ l_F= 2B \sin \alpha \sin \beta. \] 
In the special case $\alpha = \frac{\pi}{2}$, the Fagnano orbit becomes degenerate and collapses into the $2h_\alpha$ orbit. 

The doubled Fagnano orbit is a closed geodesic that belongs to a one-parameter family of closed geodesics forming a flat M\"obius strip. We represent this family by its length which is $2l_F$; see Figure \ref{doubled fagnano}

 \begin{figure}[H]  %  A=0, b=2, angle alpha = arctan(8/3), beta = arctan(8), r = 4.27... 
 	\begin{tikzpicture} [scale=0.7]
 	\draw (0,0) node[align=center, below]{A} -- (2,0)  node[align=center, below]{B}  -- (1.5,4) node[align=right, above]{C} -- (0,0) -- (-1.51, 1.1)  node[align=center, below]{B}  -- (1.5,4) --  (-2.64, 2.96) node[align=center, above]{A} -- (-1.51, 1.1);
 	\draw (-2.64, 2.96) -- (-5.08, -0.55) node[align=left, left]{C} -- (-1.51, 1.1);
 	\draw (-2.64, 2.96) -- (-4.58, 3.45) node[align=center, above]{B} -- (-5.08, -0.55) -- (-6.58, 3.45) node[align=center, above]{A} -- (-4.58, 3.45);
 	\draw[red] (2.003,0) node[align=center, below]{} -- (-4.577, 3.45) node[align=center, above]{}; 
 	\draw[blue] (1.5,0) node[align=center, below]{} -- (0.25, 0.66) ; 
 	\draw[blue] (0.25, 0.66)  -- (-5.08, 3.45) node[align=center, above]{}; 
 	\draw[blue] (1.,0) node[align=center, below]{} -- (-5.58, 3.45) node[align=center, above]{}; 
 	\draw[red] (0.613,0) node[align=center, below]{} -- (-5.967, 3.45) node[align=center, above]{}; 
 	\end{tikzpicture} 
 	\caption{The $2l_F$ family of closed geodesics, unfolded.} 
 	 \label{doubled fagnano} 
 	 \end{figure}
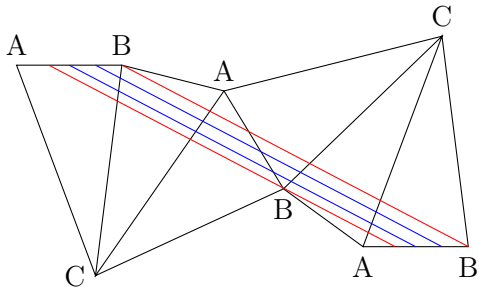

\textbf{The $2b$ orbit and its multiples.} There is a closed geodesic that we identify with its length, $2b$, created by bouncing along the top side of the trapezoid; see Figure \ref{b}. We also call the $m$-th multiple of this orbit, the $2mb$ orbit. These are diffractive periodic orbits and produce mild singularities in the wave trace. The larger the $m$, the milder the singularity. 
In our previous paper \cite{hlrn}, we took advantage of the contribution of the singularity $2b$ to prove that the Neumann spectrum determines a trapezoid among other trapezoids. 
\\

\textbf{The height $2h_\alpha$.} It corresponds to the height of the extended triangle of the trapezoid from the larger base angle $\alpha$;  see Figure \ref{h alpha}. The length of the orbit is given by
\[ 2 h_\alpha = 2B \sin \beta. \] 

\begin{figure} [H]
\begin{tikzpicture}[scale=0.8]
\draw[thick] (0,0) -- (4,0)  -- (2.5,1.5)  -- (.5,1.5) -- (0,0) ; 
\node at (2,0) [align=center, below]{$B$};
\node at (1.5,1.5) [align=center, above]{$b$};
\node at (.05, .8) [align=center]{$\ell$};
\node at (3.55, .9) [align=center]{$\ell'$};
\node at (.25, .15) [align=center,yshift=.2mm]{$\alpha$};
\node at (3.5, .15) [align=center,xshift=-.5mm,yshift=.2mm]{$\beta$};
\draw[red] (1.5,0)--(1.5,1.5);
\node at (1.7, .8) [align=center]{$h$};
\draw[dashed] (2.5,1.5)--(1,3)--(.5,1.5);
\node at (1.07, 2.62) [above,yshift=2.5mm]{$\gamma$};

\end{tikzpicture} 
\centering
\caption{A trapezoid $T$ and its extended triangle $\widehat T$.} 
\label{trapezoid} 
\end{figure}
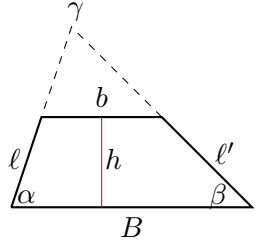%
This orbit is diffractive unless $\alpha = \frac{\pi}{2}, \frac{\pi}{3}, \text{or}\, \frac{\pi}{4}$. This is because when $\alpha = \frac{\pi}{N}$, $N \geq 5$, the height $h_\alpha$ is not inside the extended triangle as in this case $\gamma >  \pi /2$. When $\alpha = \frac{\pi}{3}$ or $\alpha =\frac{\pi}{4}$, we have $\gamma \geq \alpha \geq \beta$ therefore 
%\pink{the preceding is not correct because in our definition of a trapezoid we say $\alpha \geq \beta$} therefore
\begin{equation} \label{h < h alpha} 2h < 2h_\alpha  . \end{equation}
In the case $\alpha = \frac{\pi}{2}$, the orbit $2 h_\alpha$ is an isolated non-diffractive obit. The orbit $2h_\alpha$ is not always isolated. In fact in an isosceles trapezoid (i.e. $\alpha = \beta$) and only in this case, this orbit belongs to a cylinder of periodic orbits. See Figure \ref{c11} and the next definition. 
\\

The following family of closed geodesics exists in both trapezoids and triangles when the base angles satisfy a certain relationship.  

\begin{defn}\label{cmn}
Consider a trapezoid or a triangle when $m\alpha=n\beta\leq\pi/2$, with $m\leq n$, and  $m$ and $n$ are co-prime positive integers. Then the orbit $2 h_\alpha$ belongs to a parallel family of periodic orbits called $C_{m, n}$ that contains $2h_\alpha$ as a boundary component. These families were introduced in \cite{vgs}.  In particular, the family exists for $m=n=1$, that is for isosceles trapezoids and triangles; see Figure~\ref{c11}  \end{defn}

\begin{figure}[H]
\begin{minipage}[b][5cm][s]{.45\textwidth}
			\begin{tikzpicture}[scale=0.7]
	\draw[thick] (0,0) -- (6,0)  -- (3.691,4) --  (2.309, 4)--(0,0); 
	\draw[thick, red] (0,0)--(4.5,2.598); 
	\draw[thick, red](6,0)--(1.5, 2.598);
	\draw[blue]  (3,0)-- (0.75,1.299);
	%\draw[directed]  (3,0)  -- (0.75,1.299);
	\draw[blue] (3,0) -- (5.25, 1.299); 
	%\draw[directed] (5.25, 1.299) -- (3,0); 
	\draw[blue] (1,0) -- (0.25, 0.433); 
	\draw[blue]  (1,0)  -- (4.75, 2.165);
	\draw[blue] (2,0) -- (0.5, 0.866); 
	\draw[blue] (2,0) -- (5, 1.732);
	\draw[blue] (5,0) -- (5.75, 0.433); 
	\draw[blue]  (5,0)  -- (1.25, 2.165);
	\draw[blue] (4,0) -- (5.5, 0.866); 
	\draw[blue] (4,0) -- (1, 1.732);
	
	\node at (3, 0) [below]{};
	\node at (0,0) [left]{$\alpha$}; 
	\node at (6,0) [right] {$\beta=\alpha$}; 
	\node at (1, 1.732) [left] {}; 
	\node at (5, 1.732) [right] {}; 
	% \node at (-.2,-.2) [align=center]{$P$};
	% \node at (3.2, -.2) [align=center]{$Q$};
	%   \node at (1, -.25) [align=center]{$D$};
	%   \node at (.25, .15) [align=center]{$\alpha$};
	% \node at (2.7, .15) [align=center]{$\beta$};  
	\end{tikzpicture}
	\centering
	\caption{Unobstructed $C_{1,1}$ family.} 
	\label{c11} 
\end{minipage}
\end{figure}
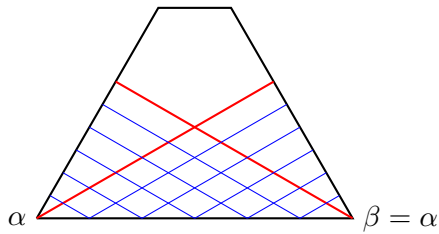

%%%%%%%%%%%%%%%%%%%%%%%%%%%%%%%%%%%%%%%%%%%%%

\subsection{Contributions of singularities in the wave trace} \label{s:classwtsing} 
Here we show that certain closed geodesics and families of closed geodesics contribute singularities to the wave trace, and we determine the order of these singularities.  We begin with the mildest singularities.  The following proposition was proved in \cite{hlrn} only in the case $m=1$. We will need that the $m$-th iterations of the $2b$ orbit produce very mild singularities, and hence can be ruled out later in the proof of our main theorem.
  
\begin{prop} \label{2mb}
	Let $T$ be a trapezoid that is not a rectangle. Let $m\in \mathbb N$.  Suppose there are no closed geodesics in $T$ of length $2mb$, other than the $m$-th iteration of the $2b$ orbit in Figure \ref{b}. Then, 
	\begin{enumerate}
	\item If $\beta \neq \frac{\pi}{2}$, 	$$I_{2mb, \mathcal B}(k) = O(k^{-m}).$$ 
	\item If $\beta = \frac{\pi}{2}$,    	   $$I_{2mb, \mathcal B}(k) = O(k^{-m/2}).$$ 
	\end{enumerate}
	In other words, the order of the singularity at $2mb$ is at most $-\frac{m}{2}$.  
	\end{prop}
\begin{proof} The proof is identical to the case $m=1$ provided in \cite{hlrn}; see Section 4 and in particular Pages 3774-3775. The only change that has to be made is that in Equation (4.1) of \cite{hlrn} we have to let $n=2m$ if $\alpha \neq \frac{\pi}{2}$, and $n=m$ if $\alpha = \frac{\pi}{2}$ and follow the same argument assuming throughout that the number of diffractions is $n$. \end{proof}

Next we study the singularity at time $\ell_F$ for the wave trace of a trapezoid in which the Fagnano triangle exists. In fact we need a more general statement on non-conical periodic orbits with an odd number of reflections. As we discussed earlier, by \cite{g1, g2} such periodic orbits are automatically isolated. 

\begin{prop} \label{Odd}
	Let $T$ be a trapezoid (or in general a polygon). Let $g$ be a periodic non-conical geodesic of length $l$ with an odd period.
	Suppose there are no other closed geodesics in $T$ of length $l$. Then as $k \to +\infty$, we have an asymptotic expansion of the form
	\[  I_{l, \mathcal{B}}(k) = e^{-ikl} \sum_{j=0}^\infty c_j(g) k^{-j}. \] 
	The constant $c_0(g)$ is nonzero.  Hence the order of the wave trace singularity at $t=l$ is $0$. 
\end{prop}

Consequently, we have 
\begin{cor} \label{OrthicSingularity}
Let $T$ be a trapezoid. Suppose the Fagnano triangle lies in $T$ and is non-diffractive as in Figure \ref{orthic}.  Suppose there are no other closed geodesics in $T$ of length $l_F$. Then the order of the wave trace singularity at $t=l_F=2B \sin \alpha \sin \beta$ is $0$. 
\end{cor}

\begin{proof}
We shall use a more general result of Guillemin-Melrose \cite{gm79}. In fact, by ~\cite{gm79}*{Theorem 1}, we obtain the proposition immediately, but we need to check that the orbit is isolated and is non-degenerate.  This means that we must verify that the linearized Poincar\'e map $P_{g}$ has no eigenvalue one, or equivalently $ \det (I - P_{g}) \neq 0$. In fact we show that
\begin{equation} \label{det}\det (I - P_{g}) =4. \end{equation}
The Poincar\'e map of a closed geodesic $g_0$ {of period $n$} is defined as follows. Let $x_0$ be a point of reflection of $g_0$ on an edge $AB$, and $\theta_0$ be the angle that $g_0$ makes with $AB$ in the counterclockwise direction. Now for $(x, \theta)$ near $(x_0, \theta_0)$ we define $f(x, \theta)$ to be the point $(x', \theta')$ in the phase space of the boundary of $T$ that is obtained by following the trajectory $g$ that starts at point $(x, \theta)$ and reflects precisely $n$ times on the boundary. In other words, $f$ is the $n$-th iterate of the billiard map. The linearized Poincar\'e map $P_g$ is the linearization (Jacobian) of $f$ at $(x_0, \theta_0)$. To calculate $P_{g_0}$ we first unfold the trapezoid along the geodesic $g$ as in Figure \ref{poincare}. The top edge $BA$ is obtained from the bottom edge $AB$ after $n$ reflections along the impact edges of $g$.

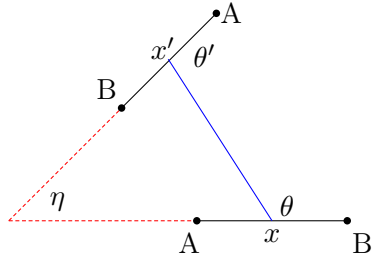
\begin{figure}[H]
	\begin{tikzpicture}
	\draw (0,0)-- (2,0) ; 
	%\draw (-2.5,4)-- (-0.5,4) ; 
	\draw (-1,1.5)-- (0.29, 2.79) ; 
	\draw [blue] (1,0)--(-0.38, 2.15);
	
	\node at (-0.1, -0.3) [align=center]{A};
	\node at (2.2, -0.3) [align=center]{B}; 
	% \node at (-2.7, 4.1) [align=center]{A}; 
	% \node at (-0.3, 4.1) [align=center]{B}; 
	\node at (-1.2, 1.75) [align=center]{B}; 
	\node at (0.46, 2.76) [align=center]{A};
	\node at (1, -.2) [align=center]{$x$};
	\node at (1.2, 0.2) [align=center]{$\theta$};
	\node at (-0.45, 2.3) [align=center]{$x'$};
	\node at (0.1, 2.2) [align=center]{$\theta'$};
	\node at (-1.85, 0.25) [align=center]{$\eta$};
	
	\draw[dash pattern=on 1.5pt off 1pt on 1.5pt off 1pt, red] (-2.5,0) -- (0, 0)   ; 
	\draw[dash pattern=on 1.5pt off 1pt on 1.5pt off 1pt, red] (-2.5,0) -- (-1, 1.5)   ;

	\node at (0,0) [circle,fill,inner sep=1pt] {};
	\node at (2,0) [circle,fill,inner sep=1pt] {};
	%\node at (-2.5,4) [circle,fill,inner sep=1pt] {};
	% \node at (-0.5,4) [circle,fill,inner sep=1pt] {};
	\node at (-1,1.5) [circle,fill,inner sep=1pt] {};
	\node at (0.26,2.76) [circle,fill,inner sep=1pt] {};
		\end{tikzpicture} 
	\vspace{.21in}
	\centering
	\caption{The Poincar\'e map}  \label{poincare} 
\end{figure}
 It is clear from Figure \ref{poincare} that
$$ x'(x, \theta_0) = 2x_0 -x, \qquad \theta'(x, \theta)= \pi +\eta - \theta= 2 \theta_0 - \theta.$$
Hence,
\begin{equation*}
P_{g_0} = 
\begin{bmatrix}
\frac{\partial x'}{\partial x} & \frac{ \partial x'}{\partial \theta}  \\[1ex]
\frac{\partial \theta'}{\partial x}  &\frac{\partial \theta'}{\partial \theta}  \\
\end{bmatrix} = \begin{bmatrix}
-1 & \frac{ \partial x'}{\partial \theta}  \\[.5ex]
0  &-1  \\
\end{bmatrix} ,
\end{equation*}
and the claim \eqref{det} follows. 
\end{proof}

Next we study the singularity associated to the height $2 h_\alpha$. This was largely done by Durso \cite{durso}. The results of \cite{durso} concern triangles but since the height $h_\alpha$ does not visit the top edge of the trapezoid, they also apply to trapezoids.  
\begin{prop} [\cite{durso}] \label{h alpha prop}
	Let $T$ be a trapezoid (or a triangle) that is not isosceles.  Suppose the height $h_\alpha$ lies inside $T$, and there are no other closed geodesics in $T$ of length $2h_\alpha$.
	\begin{enumerate}
	\item If $\alpha$ is diffractive, i.e. $\alpha \neq \frac{\pi}{N}$, $N = 2, 3, 4$, then
	$$I_{2h, \mathcal{B}}(k) = c_0(h_\alpha) e^{-2ik h_\alpha}  k^{-\frac12} + O(k^{-\frac32}),$$
	where $c_0(h_\alpha) \neq 0$. In particular the order of the singularity at $t =2 h_\alpha$ is $- \frac12$. 
	\item If $\alpha = \frac{\pi}{2}$, then
	$$ I_{l, \mathcal{B}}(k) = e^{-2ik h _\alpha} \sum_{j=0}^\infty d_j(h_\alpha) k^{-j},$$
	with $d_0(h_\alpha) \neq 0$. Thus in this case the order of the singularity is $0$.
	\end{enumerate}
	\end{prop}
See Remark \ref{c11trace} for the isosceles case. 
\begin{remark}
	Two comments are in order. Since in the cases $\alpha = \frac{\pi}{3},  \frac{\pi}{4}$, we have $h < h_\alpha$, we will not need to investigate the singularity type of $h_\alpha$. Also we note that the different singularity behavior in the non-diffractive case $\alpha =\frac{\pi}{2}$ above is not surprising, as it can be understood as the limit of Fagnano orbits collapsed into a bouncing ball as $\alpha \to \frac{\pi}{2}^-$.   
	\end{remark}

Finally we investigate the contributions of the $2h$ family in arbitrary trapezoids, and the $2h_\alpha$ family in isosceles trapezoids. 

\begin{prop} [\cite{hlrn}] \label{height}
Let $T$ be a trapezoid.  Suppose there are no other closed geodesics in $T$ of length $2h$. Then the frequency domain contribution of the $2h$ family is given by
$$I_{2h, \mathcal{B}}(k) = \frac{e^{i\pi/4}e^{-2ihk} }{\sqrt{4 \pi h}} A(R) k^\frac12+ o(k^\frac12).$$
where $A(R)=hb$ is the area of the inner rectangle of $T$. In particular, the order of the wave trace singularity is $\frac12$. 
\end{prop}

\begin{remark}\label{c11trace} The same result holds for the $C_{1,1}$ family in isosceles trapezoids, if the family is unobstructed as in Figure \ref{c11}, but $h$ must be replaced by $h_\alpha$, and $A(R)$ by half of the area that the $C_{1,1}$ family sweeps. The proof is identical to the proof we provided in \cite{hlrn}, hence we omit it. The key point is that the geometrically diffractive orbits lying on the boundary of the $C_{1,1}$ family each go through only one diffractive corner, and hence the result of \cite{hil} regarding such families on ESCS can be used in the poof of \cite{hlrn}.
\end{remark}

\section{Spectral uniqueness of a trapezoid} \label{s:proof} 

Before we present the proof of our main theorem, let us state some simple facts  (the following five propositions) that will facilitate our argument. We begin by recalling a statement from our previous work \cite{hlrn}, that specifies the length of the shortest closed geodesic in a trapezoid.  Since the proof is quite short, we include it for the convenience of the reader.  

\begin{prop}[\cite{hlrn}] \label{shortest} The length of the shortest closed geodesic in a trapezoid is either  $2h$ or $2b$.
\end{prop} 

\begin{proof} 
	Any closed diffractive or non-diffractive geodesic that starts from the top edge (including the corners) and is transversal (i.e.
	not tangent) to the top edge must be of length strictly larger than $2h$ unless the closed geodesic also runs between the two parallel sides and is a member of the $2h$ family. Furthermore, any closed geodesic that touches the left
	and right edges (including the corners) must be of length larger than $2b$ unless it is the $2b$ orbit. If a geodesic touches the bottom edge and the right edge (respectively, left edge), then it must also visit the top edge or the
	left edge (respectively, right edge) and hence its length is larger than $2h$ or
	$2b$.
\end{proof} 

\begin{prop} \label{2h or Fagnano} If we exclude the lengths of periodic orbits that lie entirely on the top edge of a trapezoid $T$ from its length spectrum, then the shortest periodic orbit is the $2h$ family or the Fagnano orbit $l_F$. 
	\end{prop}

\begin{figure} [H]
%	\centering
	\begin{minipage}{.35\textwidth}
		\begin{tikzpicture}[scale=0.45]
	\draw[thick ] (0,0) -- (6,0)  -- (2.833,3.6)  -- (1.2,3.6) --(0,0) ; 
	%\draw[dash pattern=on 1.5pt off 1pt on 1.5pt off 1pt, red] (1.8,0) -- (1.8, 3.6)   ; 
    %\draw[dash pattern=on 1.5pt off 1pt on 1.5pt off 1pt, red] (.7,1.9) -- (6,0)   ; 
	%\draw[dash pattern=on 1.5pt off 1pt on 1.5pt off 1pt, red] (3.5,2.9) -- (0,0)   ; 
	\draw[very thick,blue! 70] (3.5,2.85) -- (1.8,0)--(.67,1.93)--(3.5,2.85)   ; 
	\end{tikzpicture} 
		\centering
		\caption{ \\ $l_F$ orbit exists inside $T$.}
		\label{lF inside}
	\end{minipage}%
\begin{minipage}{.35\textwidth}
		\begin{tikzpicture}[scale=0.45]
		\draw[thick ] (0,0) -- (6,0)  -- (3.5,2.85)  -- (0.95 ,2.85) --(0,0) ; 
		%\draw[dash pattern=on 1.5pt off 1pt on 1.5pt off 1pt, red] (1.8,0) -- (1.8, 3.6)   ; 
		%\draw[dash pattern=on 1.5pt off 1pt on 1.5pt off 1pt, red] (.7,1.9) -- (6,0)   ; 
		%\draw[dash pattern=on 1.5pt off 1pt on 1.5pt off 1pt, red] (3.5,2.9) -- (0,0)   ; 
		\draw[very thick,blue! 70] (3.5,2.85) -- (1.8,0)--(.67,1.93)--(3.5,2.85)   ; 
		\draw  node at(3.5, 2.85) [draw=red,line width=0.1pt, circle,fill=blue!50,inner sep=0pt,minimum size=2pt,line width=1pt] {} ;  
		\end{tikzpicture} 
		\centering
		\caption{\\ $l_F$ orbit is {diffractive.}
		} 
		\label{lF diffracted}
	\end{minipage}%
\begin{minipage}{.35\textwidth}
	\begin{tikzpicture}[scale=0.45]
	\draw[thick ] (0,0) -- (6,0)  -- (3.95,2.3)  -- (0.75 ,2.3) --(0,0) ; 
	\draw[dash pattern=on 1.5pt off 1pt on 1.5pt off 1pt] (0,0) -- (6,0)  -- (1.66, 4.9) --(0,0) ; 
	%\draw[dash pattern=on 1.5pt off 1pt on 1.5pt off 1pt, red] (1.8,0) -- (1.8, 3.6)   ; 
	%\draw[dash pattern=on 1.5pt off 1pt on 1.5pt off 1pt, red] (.7,1.9) -- (6,0)   ; 
	%\draw[dash pattern=on 1.5pt off 1pt on 1.5pt off 1pt, red] (3.5,2.9) -- (0,0)   ; 
	\draw[thick, dash pattern=on 2pt off 1pt on 2pt off 1pt, blue] (3.5,2.85) -- (1.8,0)--(.67,1.93)--(3.5,2.85)   ; 
	\end{tikzpicture} 
	\centering
	\caption{ \\ $l_F$ does not lie in $T$. } 
	\label{lF outside}
\end{minipage}
	\end{figure}

\begin{proof} Since $2h$ is the shortest orbit, other than the $2mb$ orbits, that touches the the top edge, the proposition follows quickly from the following two claims.
	\begin{enumerate}
		\item If the $l_F$ orbit lies inside a trapezoid, then it must be the shortest geodesic that does not touch the top edge of the trapezoid.
		\item  If the $l_F$ orbit does not exist in the trapezoid or if it goes through a diffraction as in Figures \ref{lF diffracted} and \ref{lF outside}, then the $2h$ family is the shortest periodic orbit among all periodic orbits except possibly some $2mb$ orbits. 
	\end{enumerate}
 Note that if an orbit does not visit the top edge of the trapezoid then it must be an orbit of the extended triangle $\hat{T}$. Hence, the first claim follows from the classical result of Fagnano (see also \cite{durso}). For the second statement, we note that there are two cases. Either the $l_F$ orbit does not even exist in the extended triangle $\hat{T}$, or it exists in $\hat{T}$ however the trapezoid $T$ is short enough that $l_F$ does not exists in $T$ or it is diffractive as in Figures \ref{lF diffracted} and \ref{lF outside} . In the first case, the proposition follows from Durso \cite{durso}. By \eqref{lF exists}, the second case happens only if 
	\[ h \leq B \sin \beta \cos \beta. \] 
	However, since in this case $\alpha+ \beta \geq \frac{\pi}{2}$, we have $\alpha  \geq \frac{\pi}{2} - \beta$, and thus $\sin \alpha \geq \cos \beta$. Applying this to the above inequality, we obtain 
	$$ 2h \leq 2B \sin \alpha  \sin \beta = l_F.$$
	\end{proof}
 \begin{prop} \label{h alpha < lF} Let $T$ be a trapezoid. Suppose the $l_F$ orbit exists inside $T$. Then
 	$2h_\alpha < 2l_F. $
 \end{prop} 
 \begin{proof}  Since the $l_F$ orbit exists, we must have $\alpha + \beta \geq \frac{\pi}{2}$, in particular $\alpha \geq \frac{\pi}{4}$. It is then obvious that
 	\[2h_\alpha = 2B \sin \beta < 4 B \sin \alpha \sin \beta = 2l_F. \] 
 \end{proof} 
The next statement provides a useful lower bound for the length of conical periodic orbits and hence for  families of periodic orbits. 
\begin{prop} \label{conical orbits}  Any conical period orbit inside a non-rectangular trapezoid $T$ that is not a $2mb$ orbit, has length $ \geq 2h$ or $2h_\alpha$, and equality occurs if and only if the orbit belongs to the $2h$ family, or if it is the $2h_\alpha$ orbit or belongs to the $C_{1,1}$ family when $T$ is isosceles, respectively.  In particular, any family of periodic orbits has length $\geq 2h$ or $2h_\alpha$. 
\end{prop}
\begin{proof} Clearly if a conical periodic orbit goes through one of the bottom vertices and does not touch the top edge, its length is at least $2h_\alpha$. If it passes through a bottom vertex and touches the top edge, it is longer than or equal to $2h$.  If it goes through one of the top vertices it is either a $2mb$ orbit or it must be transversal to the top edge and consequently be at least $2h$ long.  The equality cases are all obvious. The second statement follows immediately because any boundary circle of a family of periodic orbits is conical. 
\end{proof}

\begin{prop}  \label{same h lF} If two trapezoids are isospectral and have the same height, then they are isometric.  If two trapezoids are isospectral and have the same $\ell_F$ and $h_\alpha$, then they are the same. 
\end{prop} 

\begin{proof} If two trapezoids are isospectral, then they have the same heat trace invariants.  Consequently they have the same area, perimeter, and angle invariant.  If in addition they have the same height, then it was proved in \cite{sos} that they are isometric.  Now let us assume that two trapezoids have the same $\ell_F$ and $h_\alpha$.   So we obtain that for trapezoids $T_1$ and $T_2$, 
	$$2B_1 \sin \beta_1 = 2 B_2 \sin \beta_2, \quad 2B_1 \sin \alpha_1 \sin \beta_1 = 2 B_2 \sin \alpha_2 \sin \beta_2.$$
	Thus $\alpha_1 = \alpha_2$ and $ \beta_1 = \beta_2$.  This further implies that $B_1=B_2$. 
   Since the trapezoids have the same perimeters,
	$$B_1+b_1 + h_1 (\csc \alpha_1 + \csc \beta_1) = B_1+b_2 + h_2(\csc \alpha_1+ \csc \beta_1). $$
	Using
	$$b_1 = B_1 - h_1(\cot \alpha_1 + \cot \beta_1), \quad b_2 = B_2 - h_2 (\cot \alpha_1 + \cot\beta_1), $$
and
	$$\csc \alpha_1 + \csc \beta_1 > \cot \alpha_1 + \cot\beta_1, $$
we obtain that  $h_1 = h_2.$ We then obtain that $b_1=b_2$, and therefore the trapezoids are isometric.  
\end{proof}

We have now demonstrated everything that we need to give the proof of our main result. 

\begin{proof}[\textbf{Proof of Theorem \ref{th:main}} ] 
	Assume that two non-obtuse trapezoids $T_1$ and $T_2$ are isospectral. We denote the parameters of $T_1$ and $T_2$ by $\alpha_1, \beta_1, b_1, B_1, h_1$ and $\alpha_2, \beta_2, b_2, B_2, h_2$, respectively.  Since by Proposition \ref{p:rect}, rectangles are spectrally unique among trapezoids, we assume that the trapezoids are not rectangles. We begin scanning the positive real line for wave trace singularities. By Proposition \ref{shortest} the first singularity is either  $2b$ or $2h$.  If the order is $\frac12$, then by Propositions \ref{height} and \ref{2mb}  it must be the $2h$ family. Then the trapezoids have the same height, and by Proposition \ref{same h lF}, they must be isometric. If the order of the first singularity is at most $- \frac12$, it must be from the $2b$ orbit. We assume this is the case and move on to the next singularity. By Proposition \ref{2h or Fagnano}, after jumping over singularities of order at most $-\frac{m}{2}$ created by the $2mb$ orbits, we arrive at $t=2h$ or $t=l_F$. If the order is $\frac12$, it must be from the $2h$ family and we are done again by Proposition \ref{same h lF}. So from now on we assume that $l_F(T_1) = l_F(T_2)$. 
		%\begin{center} 
		\begin{table} [H]
			\begin{tabular}{| c | c |c | c| c| c|}
				\hline orbit & $2mb$ & $\ell_F$ &  $2h_\alpha$ & $2h$ &  Non-conical with odd order  \\ 
				\hline  order & $\leq -m/2$ & $0$  & \vtop{\hbox{\strut $- \frac12, \,\text{if}\, \alpha \neq \beta$, $\alpha$ \text{is diffractive} }\hbox{\strut $0, \, \; \; \; \text{if}\,    \alpha = \pi/2 $ }\hbox{\strut $\frac12, \, \; \; \; \text{if}\, \alpha = \beta$}} & $ \frac12$& $0$  \\   
				\hline Proposition & \ref{2mb} & \ref{Odd} & \ref{h alpha prop} & \ref{height} & \ref{Odd} \\ \hline  
			\end{tabular}  
			
			\caption{This table summarizes the important periodic orbits for our inverse problem.} 
		\end{table} 
	%\end{center}
	We then investigate the smallest singularity of order at least $-\frac12$, call it $t_0$, in the open interval $(l_F, 2l_F)$. Note that $t_0$ cannot be $2b$ because in this case $2b < l_F$, and it cannot be $2mb$, $m \geq 2$, because their orders are $  \leq -1$ by Proposition \ref{2mb}.  We know however that $t_0$ must be the length of a conical periodic orbit or of a non-conical one with an odd number of reflections (recall that if the number of reflections is even, then the orbit belongs to a family which always contains a conical orbit on its boundary). By Proposition \ref{conical orbits}, if the singularity $t_0$ is from a conical orbit, it must be caused by the orbit $2h_\alpha$ or $2h$. If it is non-conical with an odd number of reflections, the order of the singularity must be $0$ by Proposition \ref{Odd}. Note that multiple (but only finitely many) isolated non-conical orbits may have the same length, but their singularity contributions can never cancel out or add to become a singularity of order $-\frac12$, which is the order of $2h_\alpha$, or $\frac12$, which is of $2h$ (or $2h_\alpha$ when $T$ is isosceles). This is because their $k$-expansions contain only integer powers of $k$. Note also that, although when $\alpha =\frac{\pi}{2}$, the order of the singularity at $2h_\alpha$ is $0$, but in this case $2h_\alpha = l_F$, so it does not belong to the open interval $(l_F, 2l_F)$. In short, the $k$-expansion of isolated non-conical orbits completely distinguishes them from the $2h_\alpha$ and $2h$ orbits, therefore we skip them if we encounter them.   Hence, the next singularity of nonzero order but at least $-\frac12$, call $t_1$, that occurs in the interval $(l_F, 2l_F)$ must be either from $2h_\alpha$ or $2h$. If the order of $t_1$ is $-\frac12$ we know that it must come from $2h_\alpha$, therefore $2h_{\alpha_1} = 2h_{\alpha_2}$, which implies that $T_1 =T_2$ by Proposition \ref{same h lF}. If the order of $t_1$ is $\frac12$ then one of the following cases happens. 
\begin{enumerate}
	\item 	$2h_1 = 2h_2$. 
	\item $2h_1 = 2h_{\alpha_2}$, and $T_2$ is an isosceles trapezoid, i.e. $\alpha_2 = \beta_2$. 
\end{enumerate}
If case (1) holds, we are done again by Proposition \ref{same h lF}. So suppose case (2) holds.  Since the $2h_1$ singularity of $T_1$ is observed first, this requires that 
	\begin{equation} \label{h1 < h alpha1 } 2h_1 < 2h_{\alpha_1}. \end{equation}
	Since in this case we have
	\[ 2h_1=2h_{\alpha_2} = 2 B_2 \sin \beta_2 = 2B_2 \sin \alpha_2,  \] 
 we obtain from \eqref{h1 < h alpha1 }, that
	\[ B_2 \sin \beta_2 < B_1 \sin \beta_1. \] 
	Consequently, since $\ell_F(T_1) = \ell_{F}(T_2)$, we have 
	\begin{equation*} B_1 \sin \alpha_1 \sin \beta_1 = B_2 \sin \alpha_2 \sin \beta_2 < B_1  \sin \alpha_2 \sin \beta_1. \end{equation*}
	From this we obtain $\sin\alpha_1 < \sin \alpha_2$, which implies  $\alpha_1 < \alpha_2 = \beta_2.$
	However, we also have by the heat trace the same angle invariants $q = F(\alpha)+ F(\beta)$; see Proposition \ref{qprop}. 	Since $F(x) = \frac{1}{x(\pi-x)}$ is a strictly decreasing function on the interval $(0, \frac{\pi}{2}]$, we have
	$$F(\alpha_1) + F(\beta_1) = 2 F(\alpha_2) < 2 F(\alpha_1).$$
	This shows that $F(\beta_1) < F(\alpha_1),$ but $\beta_1 \leq \alpha_1$, which contradicts that $F$ is decreasing.  Therefore, case (2) cannot happen. 
	
	The final case of concern is when no singularities of order $\neq 0$ and at least $-\frac12$ occur in the interval $(l_F, 2l_F)$. Since by Proposition \ref{h alpha < lF}, we have $2h_\alpha < 2l_F$, and since $2h_\alpha  \geq l_F$ with equality only if $\alpha =\frac{\pi}{2}$, this scenario happens only if $\alpha_1= \frac{\pi}{2}$ and $\alpha_2= \frac{\pi}{2}$. But then the angle invariant determines that $\beta_1= \beta_2$, which in turn implies that $T_1= T_2$ using the other heat trace invariants, i.e. the area and perimeter. 
	\end{proof} 

\begin{remark}
	In our proof we never considered the $2h_\alpha$ orbit in the non-diffractive cases $\alpha= \frac{\pi}{3}, \frac{\pi}{4}$. This is because by $\eqref{h < h alpha}$, in these cases $2h < 2h_\alpha$ so one would observe the singularity $2h$ sooner than $2h_\alpha$. We also did not study the obstructed $C_{1,1}$ families for the same reason. 
\end{remark}

\section*{Acknowledgements} 
The first author is supported by the Simons Foundation Collaboration Grant 638398. The second author is supported by NSF grant DMS-19-08513. The third author is supported by the {Swedish Research Council Grant} 2018-03402. The first and third authors are grateful for the support of the National Science Foundation Grant DMS-1440140 and the opportunity to work together at the Mathematical Sciences Research Institute in Berkeley, California during the Fall 2019 semester.

%%%%%%%%%%%%%%%%%%%%%%%%%%%%%%%%%

\end{document}